\documentclass[12pt]{article}
\usepackage{graphicx}
\usepackage{amsmath}
\usepackage{amsfonts}
\usepackage{amsthm}
\usepackage{url}
\usepackage[margin=1in]{geometry}
\newcommand{\ud}{\,\mathrm{d}}
\newcommand{\beqa}{\begin{eqnarray}}
\newcommand{\eeqa}[1]{\label{#1}\end{eqnarray}}
\newcommand{\beq}{\begin{equation}}
\newcommand{\eeq}[1]{\label{#1}\end{equation}}
\newcommand{\Tr}{\mathop{\rm Tr}\nolimits}

\begin{document}
\title{Explicit examples of extremal quasiconvex quadratic forms that are not polyconvex}
\author{Davit Harutyunyan and Graeme Walter Milton\\
\textit{Department of Mathematics, The University of Utah}}

\maketitle
\begin{abstract}
We prove that if the associated fourth order tensor of a quadratic form has a linear elastic cubic symmetry then it is quasiconvex if and only if it is polyconvex, i.e. a sum of convex and null-Lagrangian quadratic forms. 
We prove that allowing for slightly less symmetry, namely only cyclic and axis-reflection symmetry,
gives rise to a class of extremal quasiconvex quadratic forms, that are not polyconvex. Non-affine boundary conditions on the potential are identified which allow one to obtain sharp
bounds on the integrals of these extremal quasiconvex quadratic forms of $\nabla u$ over an arbitrary region.
\end{abstract}


\newtheorem{Theorem}{Theorem}[section]
\newtheorem{Lemma}[Theorem]{Lemma}
\newtheorem{Corollary}[Theorem]{Corollary}
\newtheorem{Remark}[Theorem]{Remark}
\newtheorem{Definition}[Theorem]{Definition}

\section{Introduction}
\label{sec:intro}

In his work in 1952, Morrey proved that convexity implies quasiconvexity which itself implies rank-one convexity,
see [\ref{bib:Mor1},\ref{bib:Mor2}]. In 1977, Ball introduced polyconvexity which was proven
to be an intermediate condition between convexity and quasiconvexity [\ref{bib:Ball1}]. For the case when the function
$f\colon\mathbb R^{N\times n}\to\mathbb R$ is quadratic, i.e., $f(\xi)=(M\xi;\xi)$ for some
symmetric matrix $M\in R^{(N\times n)\times(N\times n)} $, the rank-one convexity of $f$ is actually equivalent to its quasiconvexity,
which was established by Van Hove in [\ref{bib:VanHove1},\ref{bib:VanHove2}] and the
rank-one convexity of $f$ is equivalent to the inequality $f(x\otimes y)\geq 0$ that must hold for
all $x\in\mathbb R^N,y\in\mathbb R^n$ e.g., see [\ref{bib:Dac}, page 192]. Also polyconvexity is then equivalent
to the quadratic form being the sum of convex and null-Lagrangian quadratic forms [\ref{bib:Dac}, page 192, Lemma 5.72]. 
The case
 $n=2$ or $N=2$ received considerable attention by many authors and Terpstra [\ref{bib:Terp2}] showed
that in this case quasiconvexity implies polyconvexity.
Terpstra also showed in [\ref{bib:Terp2}] that if $n\geq 3$ or $N\geq3$, then there exist quadratic forms
that are quasiconvex but not polyconvex, but his proof did not deliver an explicit example of such
a quadratic form. An explicit example of such a quadratic form is due to Serre [\ref{bib:Ser1},\ref{bib:Ser2}], see also Ball [\ref{bib:Ball2}].
Here we will obtain a quasiconvex quadratic form which is especially easy to show that it is not polyconvex.

A special case of quasiconvex quadratic forms are the so called extremal ones introduced by Milton in [\ref{bib:Mil3}, page 87], see also [\ref{bib:Mil1}, Section 25.2].
There are 3 definitions of extremals (with only the first being introduced before), namely:
\begin{Definition}
\label{def:1}
A quadratic quasiconvex form is called an extremal if one cannot subtract a rank-one quadratic form from it while preserving the quasiconvexity of the quadratic form
\end{Definition}

\begin{Definition}
\label{def:2}
A quadratic quasiconvex form is called an extremal if one cannot subtract a quasiconvex quadratic form from it other than a multiple
of itself modulo Null-Lagrangians, while preserving the quasiconvexity of the quadratic form.
\end{Definition}
In the case $n=2$ or $N=2$ a rank-one quadratic form is itself extremal in the sense of Definition~\ref{def:2} but not in the sense of Definition~\ref{def:1}. So the two definitions are not
equivalent, but it is not known if Definition~\ref{def:1} implies Definition~\ref{def:2}.  Therefore one can use also the following third definition:

 \begin{Definition}
\label{def:3}
A quadratic quasiconvex form is called an extremal if it is an extremal in the sense of both Definition~\ref{def:1}
and Definition~\ref{def:2}.
\end{Definition}
Note that if a quadratic form is extremal in the sense of Definition~\ref{def:2} but not in the sense of Definition~\ref{def:1} then the quasiconvex quadratic form must be polyconvex
since a rank-one quadratic form is quasiconvex. So a quasiconvex quadratic form which is extremal in the sense of Definition~\ref{def:2} and which is not polyconvex is automatically
extremal in the sense of all three definitions.

Extremal quasiconvex functions occupy a privleged position of being at the boundary between null-Lagrangians and strictly quasiconvex functions. As such they share
properties with null-Lagrangians that are not shared by strictly quasiconvex functions. In particular, sharp lower bounds on the integral of $f(\nabla u)$ over a
body $\Omega$ can be obtained not just for affine boundary conditions on $u(x)$ (i.e. $u(x)=Ax$ on $\partial\Omega$, 
where $\Omega\in\mathbb R^n, A\in\mathbb R^{N\times n},$ $u\colon\overline{\Omega}\to\mathbb R^N$ and $f\colon\mathbb R^{N\times n}\to\mathbb R$), or for periodic boundary conditions on
$\nabla u$, but for other boundary conditions as well [\ref{bib:Kan.Mil},\ref{bib:Mil2}]. This is important for the following reason.
In the 1980's a lot of attention was focussed on bounding the
effective tensors of composites, and quasiconvex functions played an important role in this development: see the books [\ref{bib:Allaire},\ref{bib:Cherk},\ref{bib:Mil1},\ref{bib:Tar3}] and references therein.
In the last few years it was realized [\ref{bib:Kan.Kim.Mil}, \ref{bib:Kan.Mil}, see also \ref{bib:Mil.Ngu}, \ref{bib:Kan.Mil.Wan}, \ref{bib:Tha.Mil}, \ref{bib:Kan.Kim.Lee.Li.Mil}]
that similar methods can be useful for bounding of the Dirichlet to Neumann map of inhomogeneous bodies
(which for a two-phase body can be applied in an inverse manner to bound the volume of an inclusion from Dirichlet and Neumann boundary data).
Clearly there is an interest
in obtaining sharp bounds on the Dirchlet to Neumann map not just for affine data but for other Dirichlet boundary conditions as well. The use of null-Lagrangians,
and extremal quasiconvex functions (but not strictly quasiconvex functions) allows one to do this. It is thus important not only to find extremal quasiconvex functions,
but also to identify the special boundary conditions which allow one to obtain a sharp lower bound on the integral of $f(\nabla u)$ over $\Omega$, and to calculate this bound. Here we will present the first explicit example of an extremal quasiconvex function of $\nabla u$ which is not a null-Lagrangian, we will identify those special boundary conditions, and we will evaluate the integral  of $f(\nabla u)$ over $\Omega$ for these boundary conditions.

The identification of an explicit class of extremal quasiconvex quadratic forms is also potentially an important stepping stone. If all such extremals could be
explicitly identified then one would have an explicit characterication of all quasiconvex quadratic forms. It remains to be seen whether this can be done, but in any
case the identification of explicit and interesting quasiconvex functions is useful for deriving analytical results, as opposed to numerical results, for bounding
the response of inhomogeneous composites and bodies.

In [\ref{bib:All.Kohn}], Allaire and Kohn used extremals for strain fields (extremal in the sense that one cannot subtract a symmetrized rank-one
quadratic form from it while preserving the quasiconvexity of the quadratic form) to derive optimal bounds on the elastic energy for two-phase composites.
One can see that their example of extremals are polyconvex though.
In [\ref{bib:Mil1}, page 546], Milton provided an example of a quadratic form
$$  f(A)=\Tr(A^2)+\Tr(A^TA)-[\Tr(A)]^2 $$
 for the case $n=N=3,$ that is a quasiconvex extremal on
$3\times3$ divergence-free periodic matrix fields $A(x)$ (with its three rows, not columns, being divergence free) in the sense that the inequality
$$\langle f(A)\rangle\geq f(\langle A\rangle)$$
holds for any periodic field $A\colon\mathbb R^3\to \mathbb R^{3\times 3}$ that is divergence-free
 where the angular brackets denote the volume average over the cell of periodicity: see also [\ref{bib:Kan.Mil}]. He proved the quasiconvexity of this example by using the ideas of Murat and Tartar [\ref{bib:Tar1},\ref{bib:Tar2},\ref{bib:Mur.Tar}].
 Kang and Milton [\ref{bib:Kan.Mil}] then used this example of an extremal to
 bound the volume fraction of an inclusion in a two-phase three dimensional body,
and they obtained the following result for the integral of the quasiconvex quadratic form over a body, with non-trivial boundary conditions on the fields.
Suppose $A(x)\colon\overline{\Omega}\to\mathbb R^3,$ where $\Omega\subset\mathbb R^3,$  satisfies the boundary condition
$An=q$ at the boundary $\partial\Omega$ where $n$ is the outward normal to the boundary and $q$ has components
$$  q_\ell=\left(A^0_{\ell k}+\frac{\partial^2\alpha}{\partial x_\ell \partial x_k}-\delta_{\ell k}\Delta\alpha+\epsilon_{\ell k m}\frac{\partial\beta}{\partial x_m}\right)n_k, $$
for some scalar functions $\alpha(x)$ and $\beta(x)$ defined in the neighborhood of $\partial\Omega$,  and for some constants $A^0_{\ell k}$, where
$\epsilon_{ijm}$ is the completely antisymmetric Levi-Civita tensor taking the value $+1$ when $ijm$ is an even permutation of $123$, $-1$ when it is an odd permutation, and $0$ otherwise. Then Kang and Milton proved one has the sharp inequality
$$ \int_{\Omega}f(A(x))\,dx\geq \int_{\partial\Omega}q\cdot[(A^0+(A^0)^T-\Tr(A^0)I)x+2\nabla\alpha]\,dS $$
in which $I$ is the identity matrix, and moreover showed there is a huge range of fields $A(x)$ for which one has equality in this inequality.

For the gradient problem Milton proposed an algorithm for finding extremals in [\ref{bib:Mil2}], however no explicit example of a quasiconvex extremal quadratic form was given. The key ingredient in the algorithm
 is the following: given a quasiconvex quadratic form $f$, one tries  subtracting a rank-one positive definite
 quadratic form from $f$ such that the new quadratic form remains quasiconvex. A formula for the maximal possible coefficient of the rank-one quadratic form
 to be subtracted and the condition on the rank-one quadratic form that this coefficient be non-zero is found in [\ref{bib:Mil2}]: see also the earlier work
of [\ref{bib:Kohn.Lip}, \ref{bib:Mil3}]. The main contribution of this paper is the delivery of
 an explicit example of such an extremal for the case $n=3,N=3$ (recall that if $n\leq 2$ or $N\leq 2$ such a quadratic form does not exist).

In searching
for extremals it makes sense to first look for them amongst functions $f$ with a lot of symmetry. Then if $f(x\otimes y)$ is zero
for one pair $(x,y)$, it will also automatically be zero for all other $(x,y)$ determined by the symmetry group: this makes it likely
that $f(x\otimes y)$ has a lot of degeneracy which may make it impossible to subtract a rank-one quadratic form from it while retaining quasiconvexity.
The function $f(x\otimes y)$ which by an abuse
of notation we call $f(x,y)$ is the same for all functions $f(\xi)$ that differ by a null-Lagrangian, and we shall say
that $f$ has swap symmetry iff
\begin{equation} f(x,y)=f(y,x),
\end{equation}
cyclic symmetry iff
\begin{equation} f(x_1,x_2,x_3,y_1,y_2,y_3)=f(x_2,x_3,x_1,y_2,y_3,y_1),
\end{equation}
and axis-reflection symmetry iff
\begin{eqnarray} f(x_1,x_2,x_3,y_1,y_2,y_3) & = & f(-x_1,x_2,x_3,-y_1,y_2,y_3) \nonumber \\
& = & f(x_1,-x_2,x_3,y_1,-y_2,y_3) \nonumber \\
& = & f(x_1,x_2,-x_3,y_1,y_2,-y_3).
\end{eqnarray}
Of course there are many other types of symmetry one could consider: we could for instance have $f(x,y)=f(Ax+Cy,By+Dx)$ for all $x$ and $y$ and for all
$$ \begin{bmatrix}
A & C \\
D & B
\end{bmatrix},
$$
in some group of $6\times 6$ matrices, where the group product is the usual matrix product. The idea is to find a class of symmetries, which has sufficiently few free parameters to be amenable to analysis, yet enough
free parameters to include extremals which are not polyconvex. In section~\ref{sec:cubic.symmetry}, we first look for extremals in the class of quadratic forms with
swap, cyclic and axis-reflection symmetry: these can be associated with a rank-four tensor which has linear elastic cubic symmetry, that is determined by three parameters.
  We prove that there is no extremal in this class, i.e., the following theorem holds:
 \begin{Theorem}
\label{th:cubic.quadr}
Assume that $f(\xi)=\xi T\xi^T$ is a quadratic form, where $T$ is a linear elastic cubic symmetric rank-four tensor and $\xi=\{\xi_{ij}\}_{i,j=1}^3.$
Then if $f$ is quasiconvex it can be written as a sum of convex and Null-Lagrangian quadratic forms.
\end{Theorem}

 Then in section~\ref{sec:orthot.symmetry} we drop the swap symmetry requirement and seek an extremal in the class of quadratic forms with cyclic and axis-reflection symmetry.
These quadratic forms are determined by four parameters and among them there is an extremal:
 \begin{Theorem}
\label{th:orthotr.extremal}
The quadratic form $Q(\xi)=(\xi_{11}^2+\xi_{22}^2+\xi_{33}^2-2\xi_{11}\xi_{22}-2\xi_{22}\xi_{33}-2\xi_{33}\xi_{11})+\xi_{12}^2+\xi_{23}^2+\xi_{31}^2$
has the following properties:
\begin{itemize}
\item[(i)] $Q$ is quasiconvex,
\item[(ii)] $Q$ is not polyconvex,
\item[(iii)] $Q$ is an extremal, in all three senses of extremal.
\end{itemize}
\end{Theorem}
Furthermore we have the following Corollary
\begin{Corollary}
\label{co:orthotr.extremal}
The quadratic form $Q(\xi)=(\xi_{11}^2+\xi_{22}^2+\xi_{33}^2-2\xi_{11}\xi_{22}-2\xi_{22}\xi_{33}-2\xi_{33}\xi_{11})+\alpha\xi_{12}^2+\beta\xi_{23}^2+\gamma\xi_{31}^2$
is extremal in the sense of Definition~\ref{def:2}, where $\alpha,\beta,\gamma>0$ and $\alpha\beta\gamma=1.$
\end{Corollary}

 The problem of characterizing all such extremals is a task for the future. Finally, in section~\ref{sec:special.fields}
 we find the so called special fields introduced by Milton in [\ref{bib:Mil2}],
 for the extremals we have found, which leads to sharp inequalities on the integral of $Q(\nabla u)$ over a region $\Omega$ for certain
boundary conditions on $u$, not just affine ones.

\section{Quadratic forms with linear elastic cubic symmetry}
\label{sec:cubic.symmetry}

In this section we prove Theorem~\ref{th:cubic.quadr}, that is if the fourth-order tensor $T$ of an associated quadratic form $f$ has linear elastic cubic symmetry,
then the quasiconvexity of the quadratic form $f$ implies that $f$ is in fact a sum of convex and null-Lagrangian quadratic forms.

\textbf{Proof of Theorem~\ref{th:cubic.quadr}} If the rank four tensor $T$ has linear elastic cubic symmetry then it has 3 independent variables, namely
it satisfies the following equalities:
\begin{align*}
&T_{1111}=T_{2222}=T_{3333}=\alpha,\\
&T_{1122}=T_{2211}=T_{1133}=T_{3311}=T_{2233}=T_{3322}=\beta,\\
&T_{1212}=T_{2121}=T_{1313}=T_{3131}=T_{2313}=T_{3232}=\gamma,\\
&T_{1221}=T_{2112}=T_{1331}=T_{3113}=T_{2332}=T_{3223}=\gamma,\\
\end{align*}
and the entries that do not appear in the above equalities vanish.
Thus using the above identities we obtain for $f$ that,
\begin{eqnarray}
f(x\otimes y)&= &\sum_{i,j,k,l}T_{ijkl}x_iy_jx_ky_l \nonumber \\
&= &\alpha(x_1^2y_1^2+x_2^2y_2^2+x_3^2y_3^2)+2(\beta+\gamma)(x_1x_2y_1y_2+x_2x_3y_2y_3+x_3x_1y_3y_1) \nonumber \\
&~&+\gamma(x_1^2(y_2^2+y_3^2)+x_2^2(y_1^2+y_3^2)+x_3^2(y_1^2+y_2^2)),\\
\label{cubform}
\end{eqnarray}
for all $x=(x_1,x_2,x_3),y=(y_1,y_2,y_3)\in\mathbb R^3.$ It is clear that
$f(x\otimes y)$ can be written in the quadratic form $xT(y)x^T,$ where $T(y)$ is a $3\times3$ symmetric matrix
with entries being quadratic forms in $y.$ We will call $T(y)$ the $y$-matrix of the quadratic form $f.$
Evidently, the inequality $f(x\otimes y)\geq 0$ holds for
all $x,y\in\mathbb R^3$ if and only if the $y$-matrix $T(y)$
of $f$ is positive semi-definite for all $y\in\mathbb R^3.$ We have from the expression for $f$ that
\begin{equation}
\label{cubic.y-matrix}
T(y)=
\begin{bmatrix}
\alpha y_1^2+\gamma(y_2^2+y_3^2) & (\beta+\gamma)y_1y_2 & (\beta+\gamma)y_1y_3\\
(\beta+\gamma)y_1y_2 & \alpha y_2^2+\gamma(y_3^2+y_1^2) & (\beta+\gamma)y_2y_3\\
(\beta+\gamma)y_1y_3 & (\beta+\gamma)y_2y_3 & \alpha y_3^2+\gamma(y_1^2+y_2^2)
\end{bmatrix}.
\end{equation}
First of all from the inequalities $f((1,0,0)\otimes (1,0,0))\geq 0$ and $f((1,0,0)\otimes (0,1,0))\geq 0$
we get $\alpha\geq 0$ and $\gamma\geq0$ respectively. If $\alpha=\gamma=0,$ then obviously $f$ is quasiconvex if an only if $\beta=0$ too,
thus $f\equiv0.$ Thus we can assume without loss of generality that
\begin{equation}
\label{alpha+gamma>0}
\alpha\geq 0,\quad\gamma\geq 0,\quad \alpha+\gamma>0.
\end{equation}
From the positivity of the principal minor $M_{33}$ of $T(y)$ we get
$$M_{33}=(\alpha y_1^2+\gamma(y_2^2+y_3^2))(\alpha y_2^2+\gamma(y_3^2+y_1^2))-(\beta+\gamma)^2y_1^2y_2^2\geq 0,$$
thus the choice $y_1=y_2=1$ and $y_3=0$ gives $(\alpha+\gamma)^2\geq (\beta+\gamma)^2,$
or  $\alpha+\gamma\geq |\beta+\gamma|.$

Next, the inequality $f((1,1,1)\otimes (1,1,1))\geq 0$ gives $\beta+\gamma\geq -\frac{\alpha}{2}-\gamma,$ thus we get
\begin{equation}
\label{cubic.alpha>=beta}
-\frac{\alpha}{2}-\gamma\leq \beta+\gamma\leq \alpha+\gamma.
\end{equation}
Consider now two cases:\\
\textbf{Case1.} $\beta+\gamma\geq 0.$
In view of (\ref{cubic.alpha>=beta}) and (\ref{alpha+gamma>0}) we have $\beta+\gamma=\beta'+\gamma'$ such that $0\leq\beta'\leq \alpha$ and $0\leq\gamma'\leq \gamma,$ where $\beta'=\frac{(\beta+\gamma)\alpha}{\alpha+\gamma}$ and $\gamma'=\frac{(\beta+\gamma)\gamma}{\alpha+\gamma}$. Thus the quadratic form $f(\xi)$ can be written as

\begin{align*}
f(\xi)&=(\alpha-\beta')(\xi_{11}^2+\xi_{22}^2+\xi_{33}^2)+\beta'(\xi_{11}+\xi_{22}+\xi_{33})^2\\
&+\gamma'\left((\xi_{12}+\xi_{21})^2+(\xi_{13}+\xi_{31})^2+(\xi_{23}+\xi_{32})^2\right)\\
&+(\gamma-\gamma')(\xi_{12}^2+\xi_{21}^2+\xi_{13}^2+\xi_{31}^2+\xi_{23}^2+\xi_{32}^2)\\
&+\mathrm{Null-Lagrangian},\\
\end{align*}
as claimed.\\
\textbf{Case2.} $\beta+\gamma< 0.$ In this case again in view of (\ref{cubic.alpha>=beta}) and (\ref{alpha+gamma>0}) we have $\beta+\gamma=-(\beta'+\gamma')$ such that $0\leq\beta'\leq \frac{\alpha}{2}$ and $0\leq\gamma'\leq \gamma,$ where $\beta'=-\frac{(\beta+\gamma)\alpha}{\alpha+2\gamma}$ and $\gamma'=-\frac{2(\beta+\gamma)\gamma}{\alpha+2\gamma}$. Thus the quadratic form $f(\xi)$ can be written as
\begin{align*}
f(\xi)&=(\alpha-2\beta')(\xi_{11}^2+\xi_{22}^2+\xi_{33}^2)+\beta'\left((\xi_{11}-\xi_{22})^2+(\xi_{22}-\xi_{33})^2+(\xi_{33}-\xi_{11})^2\right)\\
&+\gamma'\left((\xi_{12}-\xi_{21})^2+(\xi_{13}-\xi_{31})^2+(\xi_{23}-\xi_{32})^2\right)\\
&+(\gamma-\gamma')(\xi_{12}^2+\xi_{21}^2+\xi_{13}^2+\xi_{31}^2+\xi_{23}^2+\xi_{32}^2)\\
&+\mathrm{Null-Lagrangian},\\
\end{align*}
as claimed.

\section{Quadratic forms with less symmetry}
\label{sec:orthot.symmetry}

Now we consider quadratic forms $f(x,y)$ with cyclic symmetry and axis-reflection symmetry. The axis-reflection symmetry ensures that terms such as $x_1^2y_1y_2$, $x_1^2y_2y_3$
and $x_1x_3y_2y_1$ cannot appear in the quadratic form. The cyclic symmetry ensures that if terms like $x_1^2y_1^2$ appear, then they must appear in the
combination $x_1^2y_1^2+x_2^2y_2^2+x_3^2y_3^2$. Such considerations imply that any quadratic form $f(x,y)$ with cyclic symmetry and axis-reflection symmetry, can be expressed as
\begin{eqnarray}
\label{cycreflect}
f(x,y)& = &a(x_1^2y_1^2+x_2^2y_2^2+x_3^2y_3^2)+b(x_1x_2y_1y_2+x_2x_3y_2y_3+x_3x_1y_3y_1) \nonumber\\
&~& +c(x_1^2y_2^2+x_2^2y_3^2+x_3^2y_1^2)+d(x_2^2y_1^2+x_3^2y_2^2+x_1^2y_3^2),
\end{eqnarray}
for some constants $a$, $b$, $c$ and $d$. If in addition $f(x,y)$ has swap symmetry then clearly $c=d$ and $f(x,y)$ has the form (\ref{cubform}) that can be associated
with a fourth order tensor $T$ having cubic symmetry. Without the swap symmetry $f(x,y)$ can be associated, modulo Null-Lagrangians, with the quadratic form
\begin{eqnarray}
f(\xi)&= & a(\xi_{11}^2+\xi_{22}^2+\xi_{33}^2)+b(\xi_{11}\xi_{22}+\xi_{22}\xi_{33}+\xi_{33}\xi_{11}) \nonumber\\
&~& +c(\xi_{12}^2+\xi_{23}^2+\xi_{31}^2)+d(\xi_{21}^2+\xi_{32}^2+\xi_{13}^2).
\label{genform}
\end{eqnarray}
The claim of Theorem~\ref{th:orthotr.extremal} is that amongst these quadratic forms, the quadratic form with $a=1$, $b=-2$, $c=1$, and $d=0$ is extremal but not polyconvex.

To establish the Corollary~\ref{co:orthotr.extremal}  we introduce the notion of \textbf{rank-one equivalence} of two quadratic forms.
\begin{Definition}
\label{def:rank-one equivalence}
Two quadratic forms $f(\xi)=\xi^TT\xi$ and $g(\xi)=\xi T'\xi^T$ are called \textbf{rank-one equivalent} if there exist
nonsingular linear transformations $A,B\colon\mathbb R^3\to \mathbb R^3$, such that
$$f(x,y)=g(Ax,By)\quad\text{for all}\quad x,y\in \mathbb R^3.$$
\end{Definition}

It is straightforward to show that this notion of equivalence is actually an equivalence relation:
\begin{Lemma}
\label{lem:rank-one.equiv}
The following properties of rank-one equivalence hold:
\begin{itemize}
\item[(i)] Any quadratic form $f$ is rank-one equivalent to itself.
\item[(ii)] If $f$ is is rank-one equivalent to $g$ then $g$ is is rank-one equivalent to $f.$
\item[(iii)] If $f$ is is rank-one equivalent to $g$ and $g$ is is rank-one equivalent to $h$ then $f$ is rank-one equivalent to $h.$
\item[(iv)]If the quadratic forms $f$ and $g$ are rank-one equivalent, then $f$ is quasiconvex if and only if $g$ is so.
\item[(v)]If the quadratic forms $f$ and $g$ are rank-one equivalent, then $f$ is an extremal quasiconvex quadratic form in the sense of Definition~\ref{def:2}
if and only if $g$ is so.
\end{itemize}
\end{Lemma}
\begin{proof}
The proof of all five properties is trivial and follows directly from the definitions of quasiconvexity and rank-one equivalence.
For $(i)$ one takes $A$ and $B$ to be the identical transformations. For $(ii)$ if $f(x,y)=g(Ax,By)$ then $g(x,y)=f(A^{-1}x,B^{-1}y).$
For $(iii)$ if $f(x,y)=g(Ax,By)$ and $g(x,y)=h(Cx,Dy),$ then $f(x,y)=h(CAx,DBy).$
 If $f$ is not quasiconvex, then $f(x,y)<0$ for some $x,y\in\mathbb R^3,$ thus $g(Ax,By)=f(x,y)<0,$ which implies that $g$ is not quasiconvex either,
thus $(iv)$ follows. If now $f$ is not an extremal in the sense of Definition~\ref{def:2}, then $f=f_1+f_2$ for some quasiconvex $f_1$ and $f_2,$ then we have
$g(x,y)=f(A^{-1}x,B^{-1}y)=f_1(A^{-1}x,B^{-1}y)+f_2(A^{-1}x,B^{-1}y)=g_1(x,y)+g_2(x,y)$ and it is evident that if $f_1$ and $f_2$ are linearly independent
then so are $g_1$ and $g_2.$

\end{proof}

Using the notion of \textbf{rank-one equivalence} and Theorem~\ref{th:orthotr.extremal} we can generate other quasiconvex quadratic forms that are extremal in the sense of Definition~\ref{def:2}. In particular
Corollary~\ref{co:orthotr.extremal} follows directly from Theorem~\ref{th:orthotr.extremal} and the following Lemma:
\begin{Lemma}
\label{lem:alpha.beta.gamma=1}
The quadratic forms $f(\xi)=(\xi_{11}^2+\xi_{22}^2+\xi_{33}^2-2\xi_{11}\xi_{22}-2\xi_{22}\xi_{33}-2\xi_{33}\xi_{11})+\alpha \xi_{12}^2+\beta \xi_{23}^2+\gamma \xi_{31}^2$ and
$g(\xi)=(\xi_{11}^2+\xi_{22}^2+\xi_{33}^2+\alpha' \xi_{12}^2-2\xi_{11}\xi_{22}-2\xi_{22}\xi_{33}-2\xi_{33}\xi_{11})+\beta' \xi_{23}^2+\gamma' \xi_{31}^2$ are equivalent if $\alpha,\beta,\gamma,\alpha',\beta',\gamma'>0$ and
$\alpha\beta\gamma=\alpha'\beta'\gamma'.$
\end{Lemma}

\begin{proof}
We will call the expression in the parentheses the principal part of the quadratic from $f.$  For any fixed $\lambda_i\neq 0,$ $i=1,2,3$ it is clear that the change of variables $x_i=\lambda_ix_i'$ and $y_i=\frac{y_i'}{\lambda_i}$ does not change the principal part of $f,$ and the sum $\alpha \xi_{12}^2+\beta \xi_{23}^2+\gamma \xi_{31}^2$ maps to $\frac{\alpha\lambda_1^2}{\lambda_2^2}x_{1}'^2y_2'^2+\frac{\beta\lambda_2^2}{\lambda_3^2}x_{2}'^2y_3'^2+\frac{\gamma\lambda_3^2}{\lambda_1^2}x_{3}'^2y_1'^2.$
Taking into account the condition $\alpha\beta\gamma=\alpha'\beta'\gamma'$ we can choose now the parameters $\lambda_1,\lambda_2$ and $\lambda_3$ such that
$f$ maps exactly to $g$ under the above variable transform.
\end{proof}

\textbf{Proof of Theorem~\ref{th:orthotr.extremal}}.
For $(i)$ we should prove that the $y$-matrix of $Q$ is positive semi-definite for all $y\in\mathbb R^3.$ By direct calculation one sees that

$$T(y)=
\begin{bmatrix}
y_1^2+y_2^2 & -y_1y_2 & -y_1y_3\\
-y_1y_2 & y_2^2+y_3^2 & -y_2y_3\\
-y_1y_3 & -y_2y_3 & y_3^2+y_1^2
\end{bmatrix}.
$$
It is easy to see that the principal minors of $T(y)$ are nonnegative:
\begin{align*}
M_{11}&=\mathrm{det}
\begin{bmatrix}
y_2^2+y_3^2 & -y_2y_3\\
-y_2y_3 & y_3^2+y_1^2
\end{bmatrix}=y_1^2y_2^2+y_1^2y_3^2+y_3^4\geq 0,\\
M_{22}&=\mathrm{det}
\begin{bmatrix}
y_1^2+y_2^2 & -y_1y_3\\
-y_1y_3 & y_3^2+y_1^2
\end{bmatrix}=y_1^2y_2^2+y_2^2y_3^2+y_1^4\geq 0,\\
M_{33}&=\mathrm{det}
\begin{bmatrix}
y_1^2+y_2^2 & -y_1y_2\\
-y_1y_2 & y_2^2+y_3^2
\end{bmatrix}=y_1^2y_3^2+y_2^2y_3^2+y_2^4\geq 0,
\end{align*}

and
$$
\det(T(y))=y_1^4y_2^2+y_2^4y_3^2+y_3^4y_1^2-3y_1^2y_2^2y_3^2\geq 0,
$$
by the Cauchy-Schwartz inequality for geometric and arithmetic means. This proves $(i).$.

The proof of $(ii)$, that $Q$ is not polyconvex is especially easy. If we assume in contradiction that $Q$ is polyconvex, we must have
\begin{equation}
\label{polyconvexity.inequality}
Q(\eta)\geq \sum_{i=1}^9\alpha_iM_i(\eta)\quad\text{for all}\quad\eta\in\mathbb R^{3\times3},
\end{equation}
where $M_i(\eta)$ is the $i$-th $2\times2$ minor of $\eta,$ see [\ref{bib:Dac}, page 192, Lemma 5.72]. Notice that the right hand side of (\ref{polyconvexity.inequality}) necessarily contains
terms that involve one of the entries $\eta_{13},\eta_{21}$ or $\eta_{31}$ unless the coefficients $\alpha_i$ are all zero. Furthermore, since the
right hand side is linear in each of these variables, whereas the left hand side does not involve them, inequality (\ref{polyconvexity.inequality})
leads to a contradiction unless all the $\alpha_i$ are zero. But we have $Q(I)=-3$ thus $Q$ is not polyconvex.

For (iii) since $Q$ is not polyconvex,
it suffices to show it is extremal in the sense of Definition~\ref{def:2}. Let us now assume that $Q$ is not extremal in this sense, i.e. it can be written as $Q_1+Q_2$ where $Q_1$ and $Q_2$ are quasiconvex quadratic
forms modulo null-Lagrangians and are linearly independent. We have that
\begin{equation}
\label{Q_1<Q}
0\leq Q_1(x\otimes y)\leq Q(x\otimes y) \quad\text{for all}\quad x,y\in\mathbb R^3.
\end{equation}
Let us prove that $\ref{Q_1<Q}$ implies then that $Q_1=\lambda Q$ for some $0\leq\lambda \leq 1.$
The proof splits into several steps.\\
\textbf{Step 1.} First we show that $Q_1(x\otimes y)$ may involve only the products $x_ix_jy_ky_l$ that can be written as products of two of the variables
$\xi_{11},\xi_{22},\xi_{33},\xi_{12},\xi_{23},\xi_{31}.$\\
\textbf{Proof of Step 1.} Let us prove that $Q_1$ cannot involve for instance the product $x_1^2y_3^2.$
Assume by contradiction that $Q_1$ involves $x_1^2y_3^2$ with a coefficient $\alpha\neq 0.$ Taking $x_2=x_3=y_1=y_2=0$ and $x_1=y_3=1$ in (\ref{Q_1<Q}) we get

$$
0\leq Q_1((1,0,0)\otimes (0,0,1))=\alpha\leq Q((1,0,0)\otimes (0,0,1))=0,
$$
thus $\alpha=0,$ which is a contradiction. Assume now by contradiction that $Q_1$ involves $x_1^2y_1y_3$ with a coefficient $\alpha\neq 0$ as a summand. Denoting the coefficient of $x_1^2y_1^2$ in $Q_1$ by $\beta$ and taking $x_2=x_3=y_2=0,$ $x_1=y_1=1$ in (\ref{Q_1<Q}) we get
$$
0\leq Q_1((1,0,0)\otimes (1,0,y_3))=\alpha y_3+\beta \leq Q((1,0,0)\otimes (1,0,y_3))=1,
$$
which cannot be satisfied for all $y_3,$ unless $\alpha=0,$ which is again a contradiction. Similarly we can prove for all other products $x_ix_jy_ky_l$ that are not a product of the variables $\xi_{ij}$ involved in $Q.$\\
Due to the fact proven in step 1 we can write for $Q_1,$
\begin{equation}
\label{Q_1}
Q_1(\xi)=\xi A\xi^T,
\end{equation}
where $A$ is a $6\times 6$ symmetric matrix and $\xi=(\xi_{11},\xi_{22},\xi_{33},\xi_{12},\xi_{23},\xi_{31}).$\\
\textbf{Step 2.} In the second step we show that
$$
A_1=(a_{ij})_{i,j=1}^3=
\begin{bmatrix}
\alpha & -\alpha & -\alpha\\
-\alpha & \alpha & -\alpha\\
-\alpha & -\alpha & \alpha
\end{bmatrix},
$$
for some $\alpha\in[0,1].$\\
\textbf{Proof of Step 2.} It is clear that the diagonal entries of both $Q_1$ and $Q-Q_1$ are nonnegative, i.e.,

\begin{equation}
\label{diagonal entries}
0\leq a_{ii}\leq 1,\qquad\text{for all}\qquad i=1,\dots,6.
\end{equation}
Due to (\ref{Q_1<Q}) we have for any $t,s\neq 0$ that
$$
0\leq Q_1((1/t,t^2s^2,0)\otimes (t^2,1/(ts),0))\leq Q((1/t,t^2s^2,0)\otimes (t^2,1/(ts),0)),
$$
thus
$$0\leq a_{11}t^2+2a_{12}t^2s+a_{22}t^2s^2+\frac{2a_{14}}{ts}+\frac{2a_{24}}{t}+\frac{2a_{44}}{t^4s^2}\leq t^2-2t^2s+t^2s^2+\frac{1}{t^4s^2}.$$

From the first inequality we obtain
$$0\leq a_{11}+2a_{12}s+a_{22}s^2+\frac{2a_{14}}{t^3s}+\frac{2a_{24}}{t^3}+\frac{2a_{44}}{t^6s^2},$$
thus sending $t$ to infinity we get $0\leq a_{11}+2a_{12}s+a_{22}s^2$ for all $s\in\mathbb R,$
which gives
\begin{equation}
\label{a11a12a22}
a_{11}a_{22}\geq a_{12}^2.
\end{equation}
Similarly we get from the second inequality that
\begin{equation}
\label{1-a11a12a22}
(1-a_{11})(1-a_{22})\geq (1+a_{12})^2.
\end{equation}
Combining (\ref{a11a12a22}) and (\ref{1-a11a12a22}) we have
\begin{align*}
\sqrt{a_{11}a_{22}}+\sqrt{(1-a_{11})(1-a_{22})}&\geq |a_{12}|+|1+a_{12}|\geq 1,\\
(1-a_{11})(1-a_{22})&\geq (1-a_{11}a_{22})^2,\\
(\sqrt{a_{11}}-\sqrt{a_{22}})^2&\leq 0
\end{align*}
thus $a_{11}=a_{22}$ and one must have equality in (\ref{a11a12a22}) and in all subsequent inequalities from (\ref{a11a12a22}), which means $a_{12}=-a_{11}.$
Similarly we get $a_{11}=a_{33}$ and $a_{13}=a_{23}=-a_{11},$ which proves step 2.
We have that
$$f(\xi)=\alpha(\xi_{11}^2+\xi_{22}^2+\xi_{33}^2-2\xi_{11}\xi_{22}-2\xi_{22}\xi_{33}-2\xi_{33}\xi_{11})+a_{44}\xi_{12}^2+a_{55}\xi_{23}^2+a_{66}\xi_{31}^2$$
$$+2a_{14}\xi_{11}\xi_{12}+2a_{15}\xi_{11}\xi_{23}+2a_{16}\xi_{11}\xi_{31}+2a_{24}\xi_{22}\xi_{12}+2a_{25}\xi_{22}\xi_{23}+2a_{26}\xi_{22}\xi_{31}$$
$$+2a_{34}\xi_{33}\xi_{12}+2a_{35}\xi_{33}\xi_{23}+2a_{36}\xi_{33}\xi_{31}+2a_{45}\xi_{12}\xi_{23}+2a_{56}\xi_{31}\xi_{23}+2a_{46}\xi_{31}\xi_{12},$$
thus the $y$-matrix of $Q_1$ will be

\begin{align*}
&T_{Q_1}(y)=\\
&\begin{bmatrix}
\alpha y_1^2+a_{44}y_2^2+2a_{14}y_1y_2 & -\alpha y_1y_2+a_{15}y_1y_3+a_{24}y_2^2+a_{45}y_2y_3 & -\alpha y_1y_3+a_{46}y_1y_2+a_{16}y_1^2+a_{34}y_2y_3\\
\cdot & \alpha y_2^2+a_{55}y_3^2+2a_{25}y_2y_3 & -\alpha y_2y_3+a_{26}y_1y_2+a_{35}y_3^2+a_{56}y_1y_3\\
\cdot & \cdot & \alpha y_3^2+a_{66}y_1^2+2a_{36}y_3y_1
\end{bmatrix}.
\end{align*}
Consider now two cases:\\
\textbf{Case 1.} $\alpha=0.$\\
\textbf{Case 2.} $\alpha>0.$\\
\textbf{Case 1.} $\alpha=0.$
In this case it is clear that the $y$-matrix could be positive semi-definite for all $y\in\mathbb R^3$ only if it has the structure
$$
T_{Q_1}(y)=
\begin{bmatrix}
a_{44}y_2^2 & a_{45}y_2y_3 & a_{46}y_1y_2\\
\cdot & a_{55}y_3^2& a_{56}y_1y_3\\
\cdot & \cdot & a_{66}y_1^2
\end{bmatrix},
$$
i.e., the quadratic form $Q_1$ depends only on the variables $\xi_{12},\xi_{23},\xi_{31}$ and is rank-one convex. Since the variables $\xi_{12},\xi_{23},\xi_{31}$
 are totally independent, $Q_1$ must be convex. Every convex quadratic form is a sum of squares of linear quadratic forms, thus $Q-(a\xi_{12}+b\xi_{23}+c\xi_{31})^2$ must be rank-one convex for some $a,b,c\in\mathbb R.$ Let us prove that this then implies that $a=b=c=0.$
 We have that
$$0=Q((-1,0,0)\otimes (-1,1,1))\geq Q((-1,1,1)\otimes (-1,1,1))-(-a+b-c)^2=-(-a+b-c)^2,$$
thus $-a+b-c=0$. Similarly $a-b-c=0$ and $-a-b+c=0,$ which implies $a=b=c=0$ and thus $Q_1\equiv 0,$ which is a contradiction.\\
\textbf{Case 2.} $\alpha>0.$ From the positivity of the $M_{33}$ minor of $T_{Q_1}(y)$ we get
\begin{equation}
\label{M33}
(\alpha y_1^2+a_{44}y_2^2+2a_{14}y_1y_2)(\alpha y_2^2+a_{55}y_3^2+2a_{25}y_2y_3)\geq (-\alpha y_1y_2+a_{15}y_1y_3+a_{24}y_2^2+a_{45}y_2y_3)^2.
\end{equation}
Substituting $y_3=0$ in (\ref{M33}) and assuming that $y_2\neq 0,$ we get,
$$(a_{44}\alpha-a_{24}^2)y_2^2\geq -2\alpha(a_{14}+a_{24})y_1y_2,$$
thus $a_{14}+a_{24}=0.$
On the other hand the coefficient of $y_1^2$ in the right hand side in (\ref{M33}) must not exceed the coefficient of $y_1^2$ in the left hand side of (\ref{M33}), which gives
$$(\alpha a_{15}-a_{55}^2)y_3^2\geq -2\alpha(a_{15}+a_{25})y_2y_3,$$
thus $a_{15}+a_{25}=0.$ Similarly, doing the same analysis for minors $M_{22}$ and $M_{11}$ we obtain,
\begin{equation*}
\label{M22M11}
\begin{cases}
a_{14}+a_{24}=0 \\
a_{15}+a_{25}=0,
\end{cases}\qquad
\begin{cases}
a_{14}+a_{34}=0 \\
a_{16}+a_{36}=0,
\end{cases}
\qquad
\begin{cases}
a_{25}+a_{35}=0 \\
a_{26}+a_{36}=0,
\end{cases}
\end{equation*}

thus the $y$-matrix of $Q_1$ has the structure

\begin{align*}
&T_{Q_1}(y)=\\
&\begin{bmatrix}
\alpha y_1^2+a y_2^2-2\beta y_1y_2 & -\alpha y_1y_2+\gamma y_1y_3+\beta y_2^2+a_{45}y_2y_3 & -\alpha y_1y_3+\delta y_1^2+\beta y_2y_3+a_{46}y_1y_2\\
\cdot & \alpha y_2^2+by_3^2-2\gamma y_2y_3 & -\alpha y_2y_3+\delta y_1y_2+\gamma y_3^2+a_{56}y_1y_3\\
\cdot & \cdot & \alpha y_3^2+cy_1^2-2\delta y_3y_1
\end{bmatrix}.
\end{align*}
Next we consider the determinant of $T_{Q_1}(y),$ that must be non-negative for all $y\in\mathbb R^3.$ It is clear that
$\det(T_{Q_1}(y))$ is a homogeneous polynomial of degree 6 in the variables $y_1,y_2$ and $y_3.$ By direct calculation (e.g. by Maple)
one sees that the highest power of each of the variables in $\det(T_{Q_1}(y))$ is 4. Moreover, the coefficients of
$y_1^4y_2^2,$ $y_2^4y_3^2$ and $y_3^4y_1^2$ in $\det(T_{Q_1}(y))$ are $-4\alpha\delta^2,$ $-4\alpha\beta^2$ and $-4\alpha\gamma^2$ respectively, thus
from the positivity of $\det(T_{Q_1}(y))$ we get,
$$-4\alpha\delta^2\geq 0,\qquad -4\alpha\beta^2\geq 0,\qquad -4\alpha\gamma^2\geq 0,$$
which implies $\beta=\gamma=\delta=0,$ thus $T_{Q_1}(y)$ has the structure
$$
T_{Q_1}(y)=
\begin{bmatrix}
\alpha y_1^2+a y_2^2& -\alpha y_1y_2+a_{45}y_2y_3 & -\alpha y_1y_3+a_{46}y_1y_2\\
\cdot & \alpha y_2^2+by_3^2 & -\alpha y_2y_3+a_{56}y_1y_3\\
\cdot & \cdot & \alpha y_3^2+cy_1^2
\end{bmatrix},
$$
therefore for $Q_1$ we get,
$$Q(\xi)=\alpha(\xi_{11}^2+\xi_{22}^2+\xi_{33}^2-2\xi_{11}\xi_{22}-2\xi_{22}\xi_{33}-2\xi_{33}\xi_{11})+a\xi_{12}^2+b\xi_{23}^2+c\xi_{31}^2
+a_{45}\xi_{12}\xi_{23}+a_{46}\xi_{12}\xi_{31}+a_{56}\xi_{23}\xi_{31}.$$
If $Q(x\otimes y)=0$ for some $x,y\in\mathbb R^3$ then $0\leq Q_1(x\otimes y)\leq Q(x\otimes y)=0,$ thus $Q_1(x\otimes y)=0.$ It is clear that
$Q((1,1,1)\otimes (1,1,1))=Q((-1,1,1)\otimes (-1,1,1)=Q((1,-1,1)\otimes (1,-1,1))=Q((1,1,-1)\otimes (1,1,-1))=0,$ thus we obtain the system
$$
\begin{cases}
-3\alpha+a+b+c+a_{45}+a_{46}+a_{56}=0\\
-3\alpha+a+b+c-a_{45}+a_{46}-a_{56}=0\\
-3\alpha+a+b+c+a_{45}-a_{46}-a_{56}=0\\
-3\alpha+a+b+c-a_{45}-a_{46}+a_{56}=0,\\
\end{cases}
$$
from which we get
\begin{equation}
\label{abc=3alpha}
a_{45}=a_{46}=a_{56}=0,\qquad a+b+c=3\alpha.
\end{equation}
We have again by direct calculation,
$$\det(T_{Q_1}(y))=(abc-4\alpha^3)y_1^2y_2^2y_3^2+\alpha(aby_2^2y_3^4+bcy_3^2y_1^4+cay_1^2y_2^4).$$
If one of the coefficients $a,b$ and $c$, say $a$ is zero, then $\det(T_{Q_1}(y))=-4\alpha y_1^2y_2^2y_3^2+\alpha bc y_3^2y_1^4$
will take negative values as $y_2\to\infty,$ thus $a,b,c>0.$ We can choose the variables
$y_1>0,$ $y_2>0$ and $y_3>0$ such that $aby_2^2y_3^4=bcy_3^2y_1^4=cay_1^2y_2^4=(a^2b^2c^2)^{1/3}y_1^2y_2^2y_3^2,$
thus we obtain
 $$\det(T_{Q_1}(y))=(abc+3\alpha(a^2b^2c^2)^{1/3}-4\alpha^3)y_1^2y_2^2y_3^2\geq 0,$$
 which gives
 \begin{equation}
 \label{abc<alpha}
 abc\geq\alpha^3.
 \end{equation}
But on the other hand we have by the Cauchy-Schwartz inequality and by (\ref{abc=3alpha}),
\begin{equation}
 \label{abc>alpha}
 3\alpha =a+b+c\geq 3(abc)^{1/3},
 \end{equation}
which will lead to a contradiction with (\ref{abc<alpha}) unless the Cauchy-Schwartz inequality turns to equality in (\ref{abc>alpha}), i.e.,
$a=b=c=\alpha.$ The last equality is nothing else but $Q_1=\alpha Q,$ thus $Q$ is indeed an extremal.\\

\subsection{The special fields of $Q$: Boundary conditions and sharp inequalities}
\label{sec:special.fields}

In this section, following Milton [\ref{bib:Mil2}], we find the special fields of the quadratic form $Q$ as well as
derive a formula for the appropriate boundary conditions on $u$, such that we obtain a sharp lower bound on integral of $Q(\nabla u)$ over a domain $\Omega$.   We will consider the quadratic form $Q$ that
appears in Theorem~\ref{sec:orthot.symmetry}.
In that case
\begin{equation}
Q(\xi)=(\xi_{11}^2+\xi_{22}^2+\xi_{33}^2-2\xi_{11}\xi_{22}-2\xi_{22}\xi_{33}-2\xi_{33}\xi_{11})+ \xi_{12}^2+\xi_{23}^2+\xi_{31}^2=\sum_{i,j=1}^3J_{ij}\xi_{ij},
\end{equation}
where the $J_{ij}$ are the elements of the matrix
\begin{equation} J=T\xi=\begin{bmatrix}
\xi_{11}-\xi_{22}-\xi_{33} & \xi_{12} & 0 \\
0 & \xi_{22}-\xi_{33}-\xi_{11} & \xi_{23} \\
\xi_{31} & 0 & \xi_{33}-\xi_{11}-\xi_{22}
\end{bmatrix},
\end{equation}
which defines the (symmetric) fourth order tensor $T$.

First we find the rank-one matrices $\xi\in\mathbb R^{3\times3}$ such that and $Q(\xi)=0.$
We have seen in section~\ref{sec:orthot.symmetry} that
$$\det(Q(y))=y_1^4y_2^2+y_2^4y_3^2+y_3^4y_1^2-3y_1^2y_2^2y_3^2=0$$
 if and only if $|y_1|=|y_2|=|y_3|,$ where $\xi=x\otimes y.$

If $y_1=y_2=y_3=z\neq 0,$ then we get
$$Q(x\otimes y)=z^2((x_1-x_2)^2+(x_2-x_3)^2+(x_3-x_1)^2)=0$$
 if and only if $x_1=x_2=x_3,$ i.e., $y=(z,z,z)$ and $x=(t,t,t).$ Assume now $y=(-z,z,z),$ where $z\neq0.$
We have that
$$Q(x\otimes y)=z^2((x_1+x_2)^2+(x_2-x_3)^2+(x_3+x_1)^2)=0$$
 if and only if $x_1=-t, x_2=x_3=t$ i.e., $y=(-z,z,z)$ and $x=(-t,t,t).$
  Therefore, by the cyclic symmetry of $Q$, all rank-one fields $\xi$ that satisfy $Q(\xi)=0$ must have one of the forms:

\begin{equation}
\label{special.rank.one.firlds}
\xi=(t,t,t)\otimes(z,z,z),\ \xi=(-t,t,t)\otimes(-z,z,z),\ \xi=(t,-t,t)\otimes(z,-z,z),\ \xi=(t,t,-t)\otimes(z,z,-z).
\end{equation}

Following the ideas in [\ref{bib:Mil2}], let us now show that
\begin{equation}
\label{quas.ineq.periodic}
\langle Q(E)\rangle\geq 0,
\end{equation}
for any field
$E\in\mathbb R^{3\times3}$ (not necessarily rank-one) that is the gradient of a periodic potential $u\colon\mathbb R^3\to\mathbb R^3$, and find the special fields of $Q,$ i.e. the ones that satisfy the equality
\begin{equation}
\label{special.fields}
\langle Q(\underbar{E})\rangle=0.
\end{equation}
Here the angular brackets mean the average over the unit cell of periodicity. Assume that $D=[-1,1]^3$ is the unit cell of
periodicity. If the cell is any rectangular parallelepiped $D$, we can achieve the situation $D=[-1,1]^3$ by change of variables shrinking or
stretching the cell, that evidently preserves the rank-one convexity of the quadratic form.
By the divergence theorem we have for any $D-$periodic potential $u\colon\mathbb R^3\to \mathbb R^3$ and the associated field $E=\nabla u,$ that
$$\langle E\rangle=\langle \nabla u\rangle=0.$$
As $Q$ is quadratic, using the idea of Murat and Tartar [\ref{bib:Mur.Tar},\ref{bib:Mor1},\ref{bib:Mor2}], (see also [\ref{bib:Mil2}]), we can write $\langle Q(E)\rangle$ in Fourier space using the Parseval's identity and we have

\begin{align*}
\langle Q(E)\rangle&=\langle Q(\nabla u)\rangle\\
&=\sum_{k=(k_1,k_2,k_3)\neq0}Q(\mathrm{Re}(\hat E(k)))+Q(\mathrm{Im}(\hat E(k)))\\
&=\sum_{k\neq0}Q(\mathrm{Re}(\hat u(k)\otimes k))+Q(\mathrm{Im}(\hat u(k)\otimes k))\\
&\geq0,
\end{align*}
due to the quasiconvexity of $Q.$ It is then clear by (\ref{special.rank.one.firlds}), that the equality holds if and only if
$\hat u(k)=0$ if $k$ does not have one of the forms $(l,l,l), (-l,l,l), (l,-l,l), (l,l,-l)$ and also
\begin{align*}
&\hat u_1(l,l,l)=\hat u_2(l,l,l)=\hat u_3(l,l,l),\\
-&\hat u_1(-l,l,l)=\hat u_2(-l,l,l)=\hat u_3(-l,l,l),\\
&\hat u_1(l,-l,l)=-\hat u_2(l,-l,l)=\hat u_3(l,-l,l),\\
&\hat u_1(l,l,-l)=\hat u_2(l,l,-l)=-\hat u_3(l,l,-l).
\end{align*}
So $\langle Q(\nabla \underbar{u})\rangle=0$ when $\underbar{u}$ takes the form
\begin{eqnarray}
\underbar{u}_1 & = &v_0(x_1+x_2+x_3)-v_1(-x_1+x_2+x_3)+v_2(x_1-x_2+x_3)+v_3(x_1+x_2-x_3),\nonumber \\
\underbar{u}_2 & = & v_0(x_1+x_2+x_3)+v_1(-x_1+x_2+x_3)-v_2(x_1-x_2+x_3)+v_3(x_1+x_2-x_3), \nonumber \\
\underbar{u}_3 & = & v_0(x_1+x_2+x_3)+v_1(-x_1+x_2+x_3)+v_2(x_1-x_2+x_3)-v_3(x_1+x_2-x_3),
\label{special.form}
\end{eqnarray}
where $v_1,v_2$ and $v_3$ are $2$-periodic $C^1$ functions defined in $\mathbb R$. For the special gradient field $\underbar{E}=\nabla \underbar{u}$ and associated special field $\underbar{J}=T\underbar{E}$ we get
\begin{eqnarray}
\label{special.field.gradient}
\underbar{E} & = & \begin{bmatrix}
v_0'+v_1'+v_2'+v_3' & v_0'-v_1'-v_2'+v_3' & v_0'-v_1'+v_2'-v_3' \\
v_0'-v_1'-v_2'+v_3' & v_0'+v_1'+v_2'+v_3' & v_0'+v_1'-v_2'-v_3' \\
v_0'-v_1'+v_2'-v_3' & v_0'+v_1'-v_2'-v_3' & v_0'+v_1'+v_2'+v_3'
\end{bmatrix}, \nonumber\\
\underbar{J} & = & \begin{bmatrix}
-v_0'-v_1'-v_2'-v_3' & v_0'-v_1'-v_2'+v_3' & 0 \\
0 & -v_0'-v_1'-v_2'-v_3' & v_0'+v_1'-v_2'-v_3' \\
v_0'-v_1'+v_2'-v_3' & 0 & -v_0'-v_1'-v_2'-v_3'
\end{bmatrix} ,
\end{eqnarray}
where $v_0', v_1',v_2'$ and $v_3'$ are the first derivatives of $v_0(t), v_1(t),v_2(t)$ and $v_3(t)$ evaluated at $t=x_1+x_2+x_3$, $t=-x_1+x_2+x_3$, $t=x_1-x_2+x_3$, and $t=x_1+x_2-x_3$ respectively.
Note that the rows of $\underbar{J}$ are divergence-free, as expected from the general theory in [\ref{bib:Mil2}].

Assume now $\Omega\in\mathbb R^3$ is a $C^1$ domain and $E(x)=\nabla u(x)\colon\Omega\to\mathbb R^{3\times3}$ satisfies the boundary
conditions $u(x)=\underbar{u}(x)$ on $\partial\Omega$ for some $\underbar{u}(x)$ of the form (\ref{special.form}) for some
continuous functions $v_i$ $i=0,1,2,3$ defined in a neighborhood of $\partial\Omega.$ Then one can extend the functions $v_i$ so that they are periodic with a unit cell of periodicity
$D$ containing $\Omega$ and define $\underbar{u}$ by (\ref{special.form}). We can extend $u(x)$ so that it is $D$-periodic and equals
$\underbar{u}$ in $D\setminus \Omega$.
We have on one hand by (\ref{quas.ineq.periodic}) that,

$$\int_{D\setminus\Omega}Q(\underbar{E}(x))\ud x+\int_{\Omega}Q(E(x))\ud x=\int_D Q(E(x))\ud x\geq 0,$$
and on the other hand since $\underbar{E}$ is a special field,
$$\int_D Q(\underbar{E}(x))\ud x=0,$$
thus we arrive at
\begin{equation}
\label{E.average>bar.E.average}
\int_{\Omega}Q(E(x))\ud x\geq\int_\Omega Q(\underbar{E}(x))\ud x.
\end{equation}
Since the rows of $\underbar{J}$ are divergence-free, we can in fact evaluate, in terms of the boundary conditions,
\begin{align*}
\int_\Omega Q(\underbar{E}(x))\ud x=\int_\Omega\sum_{i,j=1}^3\underbar{J}_{ij}\underbar{E}_{ij}\ud x=
\int_{\partial\Omega}\underbar{u}\cdot\underbar{J}n \ud S,
\end{align*}
where $n$ is the outward unit normal to $\partial\Omega$. So we obtain the inequality
$$\int_{\Omega}Q(E(x))\ud x\geq\int_{\partial\Omega}\underbar{u}\cdot\underbar{J}n \ud S, $$
which is sharp, being attained when $E(x)=\underbar{E}(x)$ inside $\Omega$.

\section*{Acknowledgements}
Bob Kohn is thanked for his interest. The authors are grateful to the National Science Foundation for support through grant DMS-1211359.

\end{document}